\documentclass[11pt]{article}
\usepackage{latexsym, amsmath, amssymb}

\usepackage{amsmath}
\usepackage{amsfonts}
\usepackage{amssymb}
\usepackage{amscd}
\usepackage{amsthm}
\usepackage{fancyhdr}
\theoremstyle{plain}
\newtheorem{theorem}{Theorem}[section]
\newtheorem{proposition}[theorem]{Proposition}
\newtheorem{remark}[theorem]{Remark}
\newtheorem{lemma}[theorem]{Lemma}
\newtheorem{definition}[theorem]{Definition}
\newtheorem{corollary}[theorem]{Corollary}
\newtheorem{example}[theorem]{Example}
\newtheorem{notation}[theorem]{Notation}
\def\T{\mathbb T}

\def\Q{\mathbb Q}

\newcommand\sF{{\mathcal F}}

\newcommand\sH{{\mathcal H}}

\newcommand\sO{{\mathcal O}}

\newcommand\sS{{\mathcal S}}

  \def \tab#1{\kern #1 truein}

  \def\E{\hbox{${\cal E}$}}
  \def\F{\hbox{${\cal F}$}}
  \def\G{\hbox{${\cal G}$}}

  \def\H{\hbox{${\cal H}$}}
  
  \def\Q{\hbox{${\cal Q}$}}

\begin{document}
 \title{Vector Bundles on Products of Varieties with $n$-blocks Collections}
\author{Edoardo Ballico and Francesco Malaspina
\vspace{6pt}\\
{\small   Universit\`a di Trento}\\
{\small\it 38050 Povo (TN), Italy}\\
{\small\it e-mail: ballico@science.unitn.it}\\
\vspace{6pt}\\
         {\small  Politecnico di Torino}\\
{\small\it  Corso Duca degli Abruzzi 24, 10129 Torino, Italy}\\
{\small\it e-mail: malaspina@calvino.polito.it}} \maketitle
\footnote{Mathematics Subject Classification 2000: 14F05, 14J60. \\
keywords: $n$-blocks Collections; vector bundles; Castelnuovo-Mumford regularity.}
  \begin{abstract} Here we consider the product of varieties with $n$-blocks collections . We
  give some cohomological splitting conditions for rank $2$ bundles. A cohomological characterization for vector bundles is also provided. The tools are  Beilinson's type spectral sequences generalized by Costa and Mir\'o-Roig. Moreover we introduce a notion of Castelnuovo-Mumford regularity on a product of finitely many
projective spaces and smooth quadric hypersurfaces in order to prove two splitting criteria for vector bundle with arbitrary rank.
  \end{abstract}

   \section*{Introduction}
A well known result by Horrocks (see \cite{Ho}) characterizes the vector bundles without intermediate cohomology on a projective space as direct sum of line bundles.
This criterion fails on  more general varieties. In fact there exist non-split vector bundles   on $X$ without intermediate cohomology. This bundles are called ACM bundles.\\

On a quadric hypersurface $\Q_n$  there is a theorem that classifies all the ACM bundles (see \cite{Kn}) as  direct sums of line bundles and spinor bundles (up to a twist - for generalities about spinor bundles see \cite{Ot2}).\\
Ottaviani has generalized Horrocks criterion to  quadrics and Grassmanniann giving cohomological splitting conditions for vector bundles (see \cite{Ot1} and \cite{Ot3}).\\
The starting point of this note is \cite{CM1} where Laura Costa and Rosa Maria Mir\'o-Roig give a  new proof of Horrocks and Ottaviani's criteria by using different techniques.
Beilinson's Theorem was stated in $1978$ and since then it has become a major tool in classifying vector bundles over projective spaces. Beilinson's spectral sequence was generalized by Kapranov (see \cite{Ka1} and \cite{Ka2}) to hyperquadrics and Grassmannians and by Costa and Mir\'o-Roig (see \cite{CM1}) to any smooth projective variety of dimension $n$ with a $n$-block collection.\\
Our aim is to use these results and techniques in order to give some cohomological splitting conditions for rank $2$ bundles on products of varieties with $n$-blocks collections.\\

 We also give  some Cohomological characterization for vector bundles. Horrocks in \cite{Ho} gives a cohomological characterization of $p$-differentials over $\mathbb{P}^n$.\\
Ancona and Ottaviani in \cite{A} give this characterization on quadrics and Costa and Mir\'o-Roig on multiprojective spaces. Here we give some different characterization on quadrics and we extend the result by Costa and Mir\'o-Roig on any product of varieties with $n$-blocks collections.\\

In the last section  we specialize on a product $X$ of finitely many
projective spaces and smooth quadric hypersurfaces.
In \cite{bm1} and \cite{bm2} we introduced a notion of Castelnuovo-Mumford regularity on quadric hypersurfaces and multiprojective spaces. We will give a suitable definition of regularity  on such a product $X$ in order to prove splitting criteria for vector bundle with arbitrary rank. Let $E$ be a vector
 bundle
on $X$. We will give two criteria 
which says when $E$ is (up to a twist)
a direct
sum of $\sO$ or the tensor product of pull-backs 
of spinor bundles on the quadric factors
of $X$ (see Theorems \ref{tpq} and \ref{spi}).\\

We want to thank Claudio Fontanari for the helpful discussions and Laura Costa and Rosa Maria Mir\'o-Roig for having showed us theirs preprints, which were the starting point of our work.

\section{Preliminaries}
Throughout the paper $X$ will be a smooth projective variety defined over the complex numbers $\mathbb{C}$ and we denote by $\mathcal{D}=D^b(\sO_X-mod)$ the derived category of bounded complexes of coherent sheaves of $\sO_X$-modules.\\
For the notations we refer to \cite{CM1}.\\
Now we give the definition of $n$-block collection in order to introduce a Beilinson's type spectral sequence generalized:
\begin{definition}
An exceptional collection $(F_0, F_1, \dots , F_m)$ of objects of $\mathcal{D}$ (see \cite{CM1} Definition $2.1.$) is a block if $Ext^i_{\mathcal{D}}(F_j,F_k)=0$ for any $i$ and $j\not=k$.\\
An $n$-block collection of type $(\alpha_0, \alpha_1, \dots , \alpha_n)$ of objects of $\mathcal{D}$ is an exceptional collection $$(\E_0, \E_1, \dots, \E_m)=(E^0_{1},\dots, E^0_{\alpha_0}, E^1_{1},\dots, E^1_{\alpha_1}, \dots, E^n_{1},\dots, E^n_{\alpha_n})$$
such that all the subcollections $\E_i=(E^i_{1},\dots, E^i_{\alpha_i})$ are blocks.
\end{definition}
\begin{notation} If $\F$ is a bundle on $X$ and $$(\E_0, \E_1, \dots, \E_m)=(E^0_{1},\dots, E^0_{\alpha_0}, E^1_{1},\dots, E^1_{\alpha_1}, \dots, E^n_{1},\dots, E^n_{\alpha_n})$$ is an $n$-block collection, we denote by $\overline{\E_j}$ the direct sum of all the bundle in $\E_j$, $$\overline{\E_j}=\oplus_{i=1}^{\alpha_j}  E^j_i.$$
So we have
$$\F\otimes\overline{\E_j}=\oplus_{i=1}^{\alpha_j} (\F\otimes E^j_i).$$
\end{notation}
\begin{example}  $(\sO_{\mathbb{P}^n}(-n), \sO_{\mathbb{P}^n}(-n+1),\dots , \sO_{\mathbb{P}^n})$ is an $n$-block collection of type $(1, 1, \dots , 1)$ on $\mathbb{P}^n$ (see \cite{CM1} Example $2.3. (1)$).
\end{example}
\begin{example}\label{e2}  Let us consider a smooth quadric hypersurface $\Q_n$ in $\mathbb P^{n+1}$.\\
We use the unified notation $\Sigma_*$ meaning that for even $n$ both the spinor bundles $\Sigma_1$ and $\Sigma_2$ are considered, and for $n$ odd, the spinor bundle $\Sigma$. We follow the notation of \cite{CM1} so the spinor bundles are twisted by $1$ with respect to those of \cite{Ot2} ($\Sigma_*=\sS_*(1)$)\\
$ (\E_0, \sO(-n+1),\dots , \sO(-1), \sO)$,
where $\E_0=(\Sigma_*(-n))$,  is an $n$-block collection of type $(1, 1, \dots , 1)$ if $n$ is odd, and of type $(2, 1, \dots , 1)$ if $n$ is even (see \cite{CM1} Example $3.4. (2)$).\\
Moreover we can have several $n$-block collections:
$$\sigma_j = (\sO(j),\dots ,\sO(n-1), \E_{n-j}, \sO(n+1),\dots , \sO(n+j-1))$$
where $\E_{n-j}=(\Sigma_*(n-1))$ and $1\leq j\leq n$ (see \cite{CM3} Proposition $4.4$).
\end{example}
\begin{example}\label{e2b}  Let $X=G(k,n)$ be the Grassmanniann of $k$-dimensional subspaces of the $n$-dimensional vector space. Assume $k>1$. We have the canonical exact sequence $$ 0 \to \sS\rightarrow \sO^n \rightarrow \Q \to 0,$$
where $\sS$ denote the tautological rank $k$ bundle and $\Q$ the quotient bundle.\\
Denote by $A(k,n)$ the set of locally free sheaves $\Sigma^{\alpha}\sS$ where $\alpha$ runs over diagrams fitting inside a $k\times (n-k)$ rectangle. We have $\Sigma^{(p,0, \dots ,   0)}\sS=S^p\sS$, $\Sigma^{(1, 1, 0, \dots ,0)}\sS=\wedge^2\sS=\sS(1)$ and $\Sigma^{(p_1+t, p_2+t, \dots p_k+t)}\sS=\Sigma^{(p_1, p_n, \dots, p_k)}\sS(t)$.\\
$A(k,n)$ can be totally ordered in such a way that we obtain a $k(n-k)$-block collection $\sigma=(\E_0,\dots , \E_{k(n-k)})$ by packing in the same block $\E_r$ the bundles $\Sigma^{\alpha}\sS$ with $|\alpha|=k(n-k)-r$ (see \cite{CM1} Example $3.4. (1)$).\\
\end{example}
\begin{example}\label{e3}  Let $X=\mathbf P^{n_1}\times\dots \times\mathbf P^{n_r}$ be a multiprojective space and $d=n_1+\dots+n_r$.\\
For any $0\leq j\leq r$, denote by $\E_j$ the collection of all line bundles on $X$ $$\sO_X(a_1^j, a_2^j, \dots , a_r^j)$$
with $-n_i\leq a_i^j\leq    0$ and $\sum_{i=1}^r a_i^j= j-d$.  Using the K\"unneth formula we prove that $\E_j$ is a block and that $$(\E_0,\dots , \E_{d})$$ is a $d$-block collection of line bundles on $X$ (see \cite{CM1} Example $3.4. (3)$).
\end{example}
Beilinson's Theorem was generalized by Costa and Mir\'o-Roig to any smooth projective variety of dimension $n$ with an $n$-block collection of coherent sheaves which generates $\mathcal{D}$ (see \cite{CM1} Theorem $3.6.$).\\

\section{Products of $n$-blocks Collections}
Let $X, Y$ be two smooth projective varieties of dimension $n$ and $m$. Let $(\G_0,\dots , \G_{n})$, $\G_i=(G_0^i,\dots , G_{\alpha_i}^i)$ be a $n$-block collection for $X$ and $(\F_0,\dots , \F_{m})$, $\F_i=(F_0^i,\dots , F_{\beta_i}^i)$ a $m$-block collection for $Y$.
\begin{notation}
We denote by $\G_i\boxtimes\F_j$ the set of all the bundles $G_k\boxtimes F_m$ on $X\times Y$ such that $G_k\in \G_i$ and $F_m\in \F_j$.\\
For any $0\leq k\leq n+m$, we define $\E_k=\G_i\boxtimes\F_j$ where $i+j=k$.
\end{notation}
\begin{lemma}$(\E_0,\dots , \E_{n+m})$ is a $(m+n)$-block collection for $X\times Y$ and it generates
the derived category $\mathcal{D}$.

\begin{proof} By \cite{CM0} Proposition $3.4.$ we have that $(\E_0,\dots , \E_{n+m})$ is a strongly exceptional collection and by \cite{CM0} Proposition $4.16.$ it generates
the derived category $\mathcal{D}$.\\
We only need to prove that each $\E_i$ is a block.
Let  $G_j\boxtimes F_k$ and $G_l\boxtimes F_m$ two different bundles of $\E_i$ we have to show that $$Ext^0( G_j\boxtimes F_k, G_l\boxtimes F_m)=Ext^0( G_j\boxtimes G_l)\otimes Ext^0 (F_k\boxtimes F_m)=0.$$
Now, since $j+k=l+m=i$ and $(j,k)\not=(l,m)$, we have $j<l$ that means $Ext^0( G_j\boxtimes G_l)=0$ or $k<m$ that means $Ext^0( F_k\boxtimes F_m)=0$.\\
All the other vanishing comes from the fact that $(\E_0,\dots , \E_{n+m})$ is a strongly exceptional collection.
\end{proof}
\end{lemma}

\begin{remark}\label{r1}Let $X=X_1\times \dots \times X_r$ where $\dim X_i=n_i$ and $d=n_1+\dots+n_r$. For every $X_i$ let us consider the $n_i$-blocks  collection $(\G_0^i, \dots \G_{n_i}^i)$. Let assume that for any $i=1, \dots r$, $\G_{n_i}^i=\{\sO\}$.\\
For any $0\leq k\leq r$, denote by $\E_k$ the collection of all  bundles on $X$ $$\G^1_{j_1}\boxtimes \dots \boxtimes \G^r_{j_r}$$
with  $\sum_{i=1}^r j_i= k$.\\
 $(\E_0,\dots , \E_d)$ is a $d$-block collection  on $X$ and it generates
the derived category $\mathcal{D}$.\\
\end{remark}
\begin{example}
  Let $X=\Q_{n_1}\times\dots \times\mathbf \Q_{n_r}$ be a multiquadric space and $d=n_1+\dots+n_r$. Let us assume $n_i\geq 3$ for every $i$.\\
  For every $\Q_{n_i}$ let us consider the $n_i$-block  collection $(\G_0^i, \dots \G_{n_i}^i)$, where $\G^i_0=(\Sigma_*(-n_i))$ and $\G_j^i=\sO(-n_i+j)$ for $1\leq j \leq n_i$.\\
  Now we define a $d$-block collection  on $X$.\\
For any $0\leq k\leq r$, denote by $\E_k$ the collection of all  bundles on $X$ $$\G^1_{j_1}\boxtimes \dots \boxtimes \G^r_{j_r}$$
with  $\sum_{i=1}^r j_i= k$.
\end{example}

On $X=X_1\times \dots \times X_r$ we can use the Beilinson's type spectral sequences generalized by Costa and Mir\'o-Roig in order to prove splitting criteria.\\
We want to apply \cite{CM1} Proposition $4.1.$ for a bundle $E$ on $X$. We need for any $i$, $0\leq i\leq d-1$,
$$H^i(X, E\otimes \E_{d-i-1})=0.$$ In particular $$H^0(X, E\otimes \E_{d-1})=0.$$
 Here $\E_d=(\sO_X(0, \dots , 0))$ but $\E_{d-1}$ is more complicated so we need the following definition:\\
\begin{definition}A bundle $E$ on $X=X_1\times \dots \times X_r$ is normalized if $H^0(X, E)\not=0$ and $H^0(X, E\otimes \E_{d-1})=0$.\\
In particular on multiquadrics $E$ is normalized if $H^0(X, E)\not=0$ and $H^0(X, E\otimes \sO_X(a_1, \dots a_r))=0$ when $a_1, \dots a_r$ are non-positive integers not all vanishing.
\end{definition}
Now we can prove the following result (cfr \cite{Ml}):\\
\begin{proposition}Let $X=X_1\times \dots \times X_r$ as in Remark \ref{r1}.  Let $E$ be  a normalized bundle on $X$.
If, for any $i$, $1\leq i\leq d-1$,
$$H^i(X, E\otimes \G^1_{j_1}\boxtimes \dots \boxtimes \G^r_{j_r})=0$$ when $\sum_{h=1}^r j_h=i+1$, then $E$ contains $\sO_X$ as a direct summand.
\begin{proof} Since $E$ is normalized $H^0(X, E)\not=0$ but $H^0(X, E\otimes \E_{d-1})=0.$ So we can apply \cite{CM1} Proposition $4.1.$ for $E$ and conclude that $E$ contains $\overline{\E_d}\cong\sO$ as a direct summand..
\end{proof}
\end{proposition}
\begin{remark}The above Proposition cannot become a splitting criterion because it is not possible to iterate the above argument. In fact we cannot obtain normalized bundles by tensoring only with $\sO(t,\dots ,t)$.
We have a splitting criteria only for rank two bundles\end{remark}
\begin{corollary} Let $X=X_1\times \dots \times X_r$ as in Remark \ref{r1}. Let $E$ be  a rank two normalized bundle on $X$.
If, for any $i$, $1\leq i\leq d-1$,
$$H^i(X, E\otimes \G^1_{j_1}\boxtimes \dots \boxtimes \G^r_{j_r})=0$$ when $\sum_{h=1}^r j_h=i+1$, then $E$ splits as a direct sum of line bundles.
\end{corollary}
\begin{remark}In the last section we will see that, when $X=\mathbf P^{n_1}\times\dots \times\mathbf P^{n_s}\times\Q_{m_1}\times\dots \times\mathbf \Q_{m_q}$ by using different techniques, it is possible to prove splitting criteria for bundles with arbitrary rank.\end{remark}

\section{Cohomological Characterization of Vector Bundles }
Horrocks in \cite{Ho} gives a cohomological characterization of $p$-differentials over $\mathbb{P}^n$.\\
Ancona and Ottaviani in \cite{A} extend this characterization on quadrics:\\
they show that, if $\sF$ is a coherent sheaf on $\Q_n$ and for some $t\in\mathbb{Z}$ and some $0<j<n$\\
$H^i(\sF(t-i-1))=0$ for $j\leq i\leq n-2$\\
 $H^i(\sF(t-i+1))=0$ for $1\leq i\leq j$\\
 $H^{n-1}(\sF\otimes \Sigma(t-n))=0$ for any spinor bundle $\Sigma$,\\
 then $\sF$ contains $\psi_j^{h^i(\sF(t-j))}$ as direct summand (see \cite{A} Theorem $6.3.$).\\
 Moreover, let $\sF$ be a coherent sheaf on $\Q_n$ and suppose that for some $t\in\mathbb{Z}$ \\
$H^i(\sF(t-i-1))=0$ for $0\leq i\leq n-2$, and\\
$H^{n-1}(\sF\otimes \Sigma(t-n))=H^{n-1}(\sF(t-n))=0$ ,
 then\\ for $n$ odd $\sF$ contains $\Sigma^\vee(-t+1)^{h^{n-1}(\sF\otimes\Sigma(t-n))}$ as direct summand\\
for $n$ even $\sF$ contains $\Sigma_1^{\vee}(-t+1)^{h^{n-1}(\sF\otimes\Sigma_2(t-n))}\oplus\Sigma_2^{\vee}(-t+1)^{h^{n-1}(\sF\otimes\Sigma_1(t-n))}$ as direct summand\\  (see \cite{A} Theorem $6.7.$).\\
\begin{remark}By using different $n$-block collections on $\Q_n$ we can obtain different cohomological characterizations.\\
Let us consider for instance  the $n$-block collection:
$$(\sO(-n+1), \E_1, \sO(-n+2),\dots , \sO)$$
where $\E_1=(\Sigma_*(-n+1))$ and its right dual $$(\sO, \psi^{\vee}_1, \dots , \psi^{\vee}_{n-2}, \sH_{n-1}, \sO(1))$$ where $\sH_{n-1}=(\Sigma_*)$, (see \cite{CM3} Proposition $4.4$).\\
Let $\sF$ be a coherent sheaf on $\Q_n$ and suppose that for some $t\in\mathbb{Z}$ \\
$H^i(\sF(t-i+1))=0$ for $1\leq i\leq n-1$, and\\
$H^{n-1}(\sF(t-n+1))=0$ ,
 then\\ for $n$ odd $\sF$ contains $\Sigma^{\vee}(-t)^{h^{n-1}(\sF\otimes\Sigma(t-n+1))}$ as direct summand\\
for $n$ even $\sF$ contains $\Sigma_1^{\vee}(-t)^{h^{n-1}(\sF\otimes\Sigma_2(t-n+1))}\oplus\Sigma_2^{\vee}(-t)^{h^{n-1}(\sF\otimes\Sigma_1(t-n+1))}$ as direct summand.\\
Moreover if $\sF$ is a coherent sheaf on $\Q_n$ and for some $t\in\mathbb{Z}$ and some $0<j<n-1$\\
$H^i(\sF(t-i-1))=0$ for $j\leq i\leq n-2$\\
 $H^i(\sF(t-i+1))=0$ for $1\leq i\leq j$\\
 $H^{n-2}(\sF\otimes \Sigma(t-n+1))=0$ for any spinor bundle $\Sigma$,\\
 $H^{n-1}(\sF(t-n+1))=0$\\
 then $\sF$ contains $\psi_j^{h^i(\sF(t-j))}$ as direct summand.
 \begin{proof}Our cohomological conditions  are corresponding to those of \cite{CM1} Proposition $4.8$ so we can conclude that $E$ contains  as a direct summand the term $E_1^{-jj}$ of the spectral sequence defined in \cite{CM1} Theorem $3.6.$\\
  $E_1^{-jj}$ is $\Sigma^{\vee}(-t)^{h^{n-1}(\sF\otimes\Sigma(t-n+1))}$ or $\Sigma_1^{\vee}(-t)^{h^{n-1}(\sF\otimes\Sigma_2(t-n+1))}\oplus\Sigma_2^{\vee}(-t)^{h^{n-1}(\sF\otimes\Sigma_1(t-n+1))}$ if $j=n-1$ and it is $\psi_j^{h^i(\sF(t-j))}$ if $0<j<n-1$.
\end{proof}
\end{remark}  Laura Costa and Rosa Maria Mir\'o-Roig give a cohomological characterization on multiprojective spaces (\cite{CM1} Proposition $4.10$ and Theorem $4.11$) and grassmannians (see \cite{CM1} Corollary $4.12.$).\\
Let us consider now in general  the products of varieties with an $n$-block collection.
We have the following result:\\
\begin{lemma}Let $X, Y$ be two smooth projective varieties of dimension $n$ and $m$. Let $(\G_0,\dots , \G_{n})$ and $(\F_0,\dots , \F_{m})$  be an $n$-block collection for $X$ and  an $m$-block collection for $Y$. Let $('\G_0,\dots , '\G_{n})$ and $('\F_0,\dots , '\F_{m})$ be their right dual collections (see \cite{CM1} Proposition $3.9.$).\\
For any $0\leq k\leq n+m$, we define $'\E_k='\G_i\otimes '\F_j$ where $i+j=k$.\\
Then $('\E_0,\dots , '\E_{n+m})$ is the right dual $(m+n)$-block collection of $(\E_0,\dots , \E_{n+m})$ on $X\times Y$.
\begin{proof}$('\E_0,\dots , '\E_{n+m})$ is an $(m+n)$-block collection by the above lemma.\\
In order to prove that it is the right dual collection of $(\E_0,\dots , \E_{n+m})$ we only need to show that it verifies the orthogonality conditions (see \cite{CM1} ($3.7$) and ($3.8$)).\\
For every couple of bundles $\G\boxtimes \F$ and $'\G\boxtimes '\F$ on $X\times Y$,
$$H^{\alpha}(X\times Y, \G\boxtimes \F\otimes '\G\boxtimes '\F)=\bigoplus_{\alpha_1+\alpha_2=\alpha}H^{\alpha_1}(X, \G\otimes '\G)\otimes H^{\alpha_2}(Y, \F\otimes  '\F).$$
So the orthogonality conditions of $('\G_0,\dots , '\G_{n})$ and $('\F_0,\dots , '\F_{m})$ on $X$ and $Y$ imply those of $('\E_0,\dots , '\E_{n+m})$ on $X\times Y$.
\end{proof}
\end{lemma}
\begin{remark}\label{r2}Let $X=X_1\times \dots \times X_r$ where $\dim X_i=n_i$ and $d=n_1+\dots+n_r$. For every $X_i$ let us consider the $n_i$-blocks  collection $(\G_0^i, \dots \G_{n_i}^i)$ and the right dual $n_i$-blocks  collection $('\G_0^i, \dots '\G_{n_i}^i)$.\\
For any $0\leq k\leq r$, denote by  $'\E_k$ the collection of all  bundles on $X$ $$'\G^1_{j_1}\boxtimes \dots \boxtimes '\G^r_{j_r}$$
with  $\sum_{i=1}^r j_i= k$.\\
$('\E_0,\dots , '\E_d)$ is the right dual $d$-block collection of $(\E_0,\dots , \E_d)$.\\
The block $'\E_{d-j}$ is the orthogonal block of the block $\E_{j}$ because it contains all the orthogonal bundles of the bundles in $\E_{j}$.
\end{remark}
\begin{example}Let us consider   $X=\mathbf P^{n_1}\times\dots \times\mathbf P^{n_r}\times\Q_{m_1}\times\dots \times\mathbf \Q_{m_s}$ with $d=n_1+\dots +n_r+m_1+\dots m_s$. On every $\mathbf P^{n_1}$ we consider the $n_i$-block collection:
$$\G^{n_i}=(\sO(-n_i), \sO(-n_i+1),\dots , \sO)$$
and its right dual $$(\sO, \T(-1),\dots , \wedge^{n_i}\T(-n_i))$$ and on every  $\Q_{m_i}$ we consider the $m_i$-block collection:
$$\H^{m_i}=(\sO(-m_i+1), \sH^{m_i}_1, \sO(-m_i+2),\dots , \sO)$$
where $\H^{m_i}_1=(\Sigma_*(-m_i+1))$ and its right dual $$(\sO, \psi^{\vee}_1, \dots , \psi^{\vee}_{m_i-2}, '\sH^{m_i}_{m_i-1}, \sO(1))$$ where $'\sH^{m_i}_{m_i-1}=(\Sigma_*)$.\\
We can obtain a $d$-block collection $$(\E_0, \dots \E_d)$$ as in Remark \ref{r1}.\\
For any $\sO(t_1)\boxtimes\dots\boxtimes \sO(t_r)\boxtimes\sO(k_1)\boxtimes\dots\boxtimes\sO(k_{h})\boxtimes\Sigma_*(-m_{h+1}-1)\boxtimes\dots\boxtimes\Sigma_*(-m_{s}-1)\in\E_{d-k}$ and any $0\leq k\leq d$ ($0\leq h\leq 0)$.\\
$$R_{\E_d\dots\E_{d-k+1}}\sO(t_1)\boxtimes\dots\boxtimes \sO(t_r)\boxtimes\sO(k_1)\boxtimes\dots\boxtimes\sO(k_{h})\boxtimes\Sigma_*(-m_{h+1}-1)\boxtimes\dots\boxtimes\Sigma_*(-m_{s}-1)=$$  $$\wedge^{-t_1}\T(t_1)\boxtimes\dots\boxtimes \wedge^{-t_r}\T(t_r)\boxtimes\psi^{\vee}_{-k_1}(k_1)\boxtimes\dots\boxtimes\psi^{\vee}_{-k_h}(k_{h})\boxtimes\Sigma_*\boxtimes\dots\boxtimes\Sigma_*$$
\end{example}
\begin{proposition}Let $X=X_1\times \dots \times X_r$ where $\dim X_i=n_i$ and $d=n_1+\dots+n_r$. Let $(\E_0,\dots , \E_d)$ be the  $d$-block collection with right dual $('\E_0,\dots , '\E_d)$ as in Remark \ref{r2}.
We denote by $r_j$ the rank of $\overline{'\E_j}$.\\
Let $E$ be a bundle on $X$.\\
 Assume there exists $j$, $0<j<n$ such that for any $-n\leq p\leq -j-1$
$$H^{-p-1}(X, E\otimes \overline{\E_{p+n}})=0$$
and for any $j+1\leq p\leq 0$
$$H^{-p+1}(X, E\otimes \overline{\E_{p+n}})=0.$$
If rank $E=r_{d-j}$, then $$E\cong \overline{'\E_{d-j}}.$$
\begin{proof}Our cohomological conditions  are corresponding to those of \cite{CM1} Proposition $4.8$ so we can conclude that $E$ contains  as a direct summand the term $E_1^{-jj}$ of the spectral sequence defined in \cite{CM1} Theorem $3.6.$\\ Now, since rank $E=r_{d-j}$, then $$E\cong \overline{'\E_{d-j}}.$$
\end{proof}
\end{proposition}
\begin{example} Let us consider $X=\mathbf P^3 \times \Q_3$.\\  On $\mathbf P^3$ we have the collection $(\sO(-3),\dots , \sO)$
and the right dual $(\sO, \T(-1),\dots , \wedge^{3}\T(-3))$ and on   $\Q_3$ we have the collection
$(\Sigma(-3)), \sO(-2),\dots , \sO)$ and the right dual $(\sO, \psi^{\vee}_1, \psi^{\vee}_2, \Sigma)$.\\
We can obtain a $6$-block collection $$(\E_0, \dots \E_6)$$ as in Remark \ref{r1}.\\
Let $E$ be a rank $12$ bundle such that $$H^1(X, E\otimes \overline{\E_{6}})=H^1(X, E)=0$$
and for any $1\leq i\leq 5$
$$H^{i}(X, E\otimes \overline{\E_{5-i}})=0.$$
Then $$E\cong (\T(1)\boxtimes \sO)\oplus (\sO\boxtimes\psi_1)\cong \T(1)\boxtimes \psi_1.$$
\end{example}

\section{Regularity on $\mathbf P^{n_1}\times\dots \times\mathbf P^{n_s}\times\Q_{m_1}\times\dots \times\mathbf \Q_{m_q}$}
In \cite{bm1} we introduced the following definition of regularity on $\Q_n$ (cfr \cite{bm1} Definition $2.1$ and Proposition $2.4$):
\begin{definition} A  coherent sheaf $F$ on $\Q_n$ ($n\geq 2$) is said to be $m$-Qregular if $H^i(F(m-i))=0$ for $i=1,\dots, n-1$, $H^{n-1}(F(m)\otimes \Sigma_*(-n+1))=0$ and
    $H^n(F(m-n+1))=0$.\\
  We will say Qregular in order to $0$-Qregular.
  \end{definition}
  In \cite{bm2} we introduced the following definition of regularity on $\mathbf P^{n_1}\times\dots \times\mathbf P^{n_s}$ (cfr \cite{bm2} Definition $4.1$):
\begin{definition}\label{d4}
A  coherent sheaf $F$ on $\mathbf P^{n_1}\times\dots \times\mathbf P^{n_s}$  is said to be {\it $(p_1, \dots,
p_s)$-regular} if, for all $i>0$,
$$H^i(F(p_1, \dots, p_s)\otimes \sO(k_1, \dots, k_s))=0$$ whenever $k_1+ \dots, +k_s=-i$ and $-n_j\leq k_j\leq 0$ for any
$j=1,\dots , s$.\\
\end{definition}
Now we want to introduce a notion of regularity on $\mathbf P^{n_1}\times\dots \times\mathbf P^{n_s}\times\Q_{m_1}\times\dots \times\mathbf \Q_{m_q}$.\\
 Let us consider $X=\mathbf P^{n}\times\Q_{m}$.\\
\begin{definition}\label{dn}
On $\mathbf P^{n}$ we consider the $n$-block collection:
$$(\E_0, \dots \E_n)=(\sO(-n), \sO(-n+1),\dots , \sO)$$
and  on   $\Q_{m}$ we consider the $m$-block collection:
$$(\G_0, \dots \G_m)=(\sO(-m+1), \G_1,\dots , \sO)$$
where $\G_1=(\Sigma_*(-m+1))$.\\
A  coherent sheaf $F$ on $X$ is said to be $(p,p')$-regular if, for all $i>0$, $$H^i(F(p,p')\otimes \E_{n-j}\boxtimes\G_{m-k})=0$$ whenever $j+k=i$, $-n\leq -j\leq 0$ and $-m\leq -k\leq 0$.\\
\end{definition}
\begin{remark}\label{r22} If $m=2$  Definition \ref{dn} coincides with Definition \ref{d4} on $\mathbf P^{n}\times\mathbf P^1\times\mathbf P^1$.\\
In fact the $2$-block collection on $Q_2$ is $$(\sO(-1), \{\Sigma_1(-1), \Sigma_2(-1)\}, \sO)=(\sO(-1,-1), \{\sO(-1,0), \sO(0,-1)\}, \sO).$$
\end{remark}
\begin{remark} If $m=0$ we can identify $X$ with $\mathbf P^{n}$ and the sheaf $F(k,k')$ with $F(k)$. Under this
 identification $F$ is $(p,p')$-regular in the sense of Definition \ref{dn}, if and only if $F$ is $p$-regular
 in the sense of Castelnuovo-Mumford.\\
In fact, let $i>0$, $H^i(F(p,p')\otimes \E_{n-j}\boxtimes\G_{m-k})=H^i(F(p-j))=0$ whenever $j+k=i$, $-n\leq -j\leq 0$ and $-m\leq -k\leq 0$
if and only if $H^i(F(p-j))=0$ whenever $-i\leq -j\leq 0$ if and only if $H^i(F(p'-i))=0$.\\
If $n=0$ we can identify $X$ with $\Q_{m}$ and the sheaf $F(k,k')$ with $F(k')$. Under this
 identification $F$ is $(p,p')$-regular in the sense of Definition \ref{dn}, if and only if $F$ is $p'$-Qregular
 on $\Q_m$.\\
In fact, let $i>0$, $H^i(F(p,p')\otimes \E_{n-j}\boxtimes\G_{m-k})=H^i(F(p')\otimes\G_{m-k})=0$ whenever $j+k=i$, $-n\leq -j\leq 0$ and $-m\leq -k\leq 0$
if and only if $H^i(F(p')\otimes\G_{m-i})=0$  if and only if $F$ is $p'$-Qregular.
\end{remark}
\begin{lemma} $(1)$ Let $H$ be a generic hyperplane of $\mathbf P^{n}$. If $F$ is a regular coherent sheaf on $X=\mathbf P^{n}\times\Q_{m}$,
then $F_{|L_1}$ is  regular on $L_1=H\times \Q_{m}$.\\
$(2)$ Let $H'$ be a generic hyperplane of $\Q_m$. If $F$ is a regular coherent sheaf on $X=\mathbf P^{n}\times\Q_{m}$,
then $F_{|L_2}$ is  regular on $L_2= \mathbf P^{n}\times H'$.\\
\end{lemma}
\begin{proof} $(1)$ We follow the proof of \cite{hw} Lemma $2.6.$. We get this exact cohomology sequence:
$$H^i(F(-j,0)\otimes \sO\boxtimes\G_{m-k}) \rightarrow H^i(F_{|L_1}(-j,0)\otimes \sO\boxtimes\G_{m-k}) \rightarrow H^{i}(F(-j-1,0)\otimes \sO\boxtimes\G_{m-k})$$
If $j+k=i$, $-n\leq -j\leq 0$ and $-m\leq -k\leq 0$, we have also $-n-1\leq j-1\leq 0$, so the first and the
third groups vanish  by hypothesis. Then also the middle group vanishes and $F_{|L_1}$ is regular.\\
$(2)$ We have to deal also with the spinor bundles. Firs assume $m$ even, say $m=2l$. We have ${\Sigma_1}_{|\Q_{m-1}}\cong {\Sigma_2}_{|\Q_{m-1}}\cong \Sigma$. Let us consider  the exact sequences 
$$ 0\rightarrow \sO(-j)\boxtimes \Sigma_1(-m)\rightarrow \sO(-j)\boxtimes\sO(-m+1)^{2^l} \rightarrow \sO(-j)\boxtimes\Sigma_2(-m+1)\rightarrow 0$$ tensored by $F$.\\
Let $i\geq m-1$ and $j=m-1-j$.
Since $H^{i}(F\otimes\sO(-j)\boxtimes \Sigma_2(-m+1))=H^i(F\otimes \E_{n-j}\boxtimes\G_1)=0$  and $H^{i+1}(F(-j,-m+1))=H^{i+1}(F\otimes \E_{n-j}\boxtimes\G_0)=0$, we also have\\ $H^{i+1}(F\otimes \sO(-j)\boxtimes \Sigma_1(-m))=0$.\\ From the exact sequences
$$0\rightarrow \sO(-j)\boxtimes\Sigma_1(-m+1)\rightarrow \sO(-j)\boxtimes\Sigma_1(-m+2) \rightarrow \sO(-j)\boxtimes{\Sigma_1}_{|\Q_{m-1}}(-m+2)\rightarrow 0$$ tensored by $F$, we get $$ H^i(F(-j,0)\boxtimes\Sigma_1(-m+1)) \rightarrow H^i(F(-j,0)\boxtimes{\Sigma_1}_{|\Q_{m-1}}(-m+1))\rightarrow H^{i+1}(F(-j,0)\boxtimes\Sigma_1(-m))$$
If $i\geq m-1$ and $j=m-1-j$, the first and the
third groups vanish  by hypothesis. Then also the middle group vanishes.\\
Assume now $m$ odd, say $m=2l+1$. We have ${\Sigma}_{|\Q_{m-1}}\cong \Sigma_1\oplus\Sigma_2$. We can consider  the exact sequences 
$$ 0\rightarrow \sO(-j)\boxtimes \Sigma(-m)\rightarrow \sO(-j)\boxtimes\sO(-m+1)^{2^{l+1}} \rightarrow \sO(-j)\boxtimes\Sigma(-m+1)\rightarrow 0$$ tensored by $F$. Then we argue as above.\\
 All the others vanishing in Definition \ref{dn} can be proved as in $(1)$ and we can conclude that $F_{|L_2}$ is regular.\\
\end{proof}
\begin{proposition}\label{p1} Let $F$ be a regular coherent sheaf on $X=\mathbf P^{n}\times\Q_{m}$ then
  \begin{enumerate}
  \item $F(p,p')$ is regular for $p,p'\geq 0$.\\
  \item $H^0(F(k,k'))$ is spanned by $$H^0(F(k-1,k'))\otimes H^0(\sO(1,0))$$ if $k-1, k'\geq 0$; and it is
spanned by
  $$H^0(F(k,k'-1))\otimes H^0(\sO(0,1))$$ if $k, k'-1\geq 0$ and $m>2$.\\
  \end{enumerate}
  \end{proposition}

  \begin{proof} $(1)$ We will prove part $(1)$ by induction. Let $F$ be a regular coherent sheaf, we want show that also $F(1,0)$ is regular.  We follow the proof of \cite{hw} Proposition $2.7.$\\
   Consider the exact
   cohomology sequence:
$$H^i(F(-j,0)\otimes \sO\boxtimes\G_{m-k}) \rightarrow H^i(F(-j+1,0)\otimes \sO\boxtimes\G_{m-k}) \rightarrow H^{i}(F_{|L_1}(-j+1,0)\otimes \sO\boxtimes\G_{m-k})$$
If $j+k=i$, $-n\leq -j\leq 0$ and $-m\leq -k\leq 0$, so the first and the third groups vanish  by hypothesis.
Then also the middle group vanishes.\\
A symmetric argument  shows the vanishing for $F(0,1)$. We only have to check the vanishing involving the spinor bundles.\\
$$H^{i}(F(-j,0)\boxtimes\Sigma_*(-m+1))\rightarrow  H^i(F(-j,1)\boxtimes\Sigma_*(-m+1)) \rightarrow H^i(F(-j,1)\boxtimes{\Sigma_*}_{|\Q_{m-1}}(-m+1))$$
If $i\geq m-1$ and $j=m-1-j$, the first and the
third groups vanish  by hypothesis. Then also the middle group vanishes.\\
$(2)$ We will follow the  proof of \cite{hw} Proposition $2.8.$\\
We consider the following diagram:
$$  \begin{array}{ccc}
 H^0(F(k-1, k'))\otimes H^0(\sO(1,0))& \xrightarrow{\sigma}& H^0(F_{|L_1}(k-1,k'))\otimes
 H^0(\sO_{L_1}(1,0))\\
 \downarrow \scriptsize{\mu}& &\downarrow \scriptsize{\tau}\\
 H^0(F(k,k'))&\xrightarrow{\nu}& H^0(F_{|L_1}(k,, k'))

 \end{array}$$
  Note that $\sigma$ is surjective if $k-1, k'\geq 0$ because $H^1(F(k-2,k'))=0$
  by regularity.\\
 Moreover also $\tau$ is surjective by $(2)$ for $F_{|L_1}$.\\
  Since both $\sigma$ and $\tau$ are surjective, we can see as in \cite{m} page $100$ that $\mu$ is also
  surjective.\\
  In order to prove that $H^0(F(k,k'))$ is spanned by $H^0(F(k,k'-1))\otimes H^0(\sO(0,1))$ if $k, k'-1\geq 0$, we can use a symmetric argument since for $m>2$ the spinor bundles are not involved in the proof.
  \end{proof}

  \begin{remark}\label{gg} If  $F$ is a regular coherent sheaf on $X=\mathbf P^{n}\times\Q_{m}$ ($m>2$) then it is globally generated.\\
  In fact by the above proposition we have the following surjections:
  $$H^0(F)\otimes
H^0(\sO(1,0))\otimes H^0(\sO(0,1))\rightarrow H^0(F(1,0))\otimes H^0(\sO(0,1)) \rightarrow H^0(F(1,1)),$$ and so
the map
$$H^0(F)\otimes H^0(\sO(1,1))\rightarrow H^0(F(1,1))$$ is a
  surjection.\\
  Moreover we can consider a sufficiently large twist $l$ such that $F(l,l)$ is globally generated. The
  commutativity of the diagram
$$\begin{matrix}
H^0(F)\otimes H^0(\sO(l,l))\otimes\sO
&\to&H^0(F(l,l))\otimes\sO\\
\downarrow&&\downarrow\\
H^0(F)\otimes\sO(l,l)&\to&F(l,l)
\end{matrix}$$
yields the surjectivity of $H^0(F)\otimes\sO(l,l)\to F(l,l)$, which implies that $F$ is generated by its
sections.\\
If $m=2$, then $F$ is spanned by Remark \ref{r22} and \cite{bm2} Remark $2.6.$ 
  \end{remark}

      Now we generalize Definition \ref{dn}:
      
      \begin{definition}\label{dq} Let us consider $X=\mathbf P^{n_1}\times\dots \times\mathbf P^{n_s}\times\Q_{m_1}\times\dots \times\mathbf \Q_{m_q}$.\\
On $\mathbf P^{n_j}$ (where $j=1, \dots, s$) we consider the $n_j$-block collections:
$$(\E^j_0, \dots \E^j_n)=(\sO(-n_j), \sO(-n_j+1),\dots , \sO)$$
and  on   $\Q_{m^l}$ (where $l=1, \dots, q$) we consider the $m_q$-block collections:
$$(\G^l_0, \dots \G^l_m)=(\sO(-m_l+1), \G^l_1,\dots , \sO)$$
where $\G^l_1=(\Sigma_*(-m_l+1))$.\\
A  coherent sheaf $F$ on $X$ is said to be $(p_1,\dots ,p_{s+q})$-regular if, for all $i>0$, $$H^i(F(p_1,\dots ,p_{s+q})\otimes \E^1_{n_1-k_1}\boxtimes\dots\boxtimes\E^s_{n_s-k_s}\boxtimes\G^1_{m_1-h_{1}}\boxtimes\dots\boxtimes\G^q_{m_q-h_q}))=0$$ whenever $k_1+\dots+k_s+h_1+\dots +h_q=i$, $-n_j\leq -k_j\leq 0$ for any
$j=1,\dots , s$ and $-m_l\leq -h_l\leq 0$ for any
$l=1,\dots , q$.\\ 
\end{definition}
\begin{remark}All the above properties (in particular Remark \ref{gg}) can be proved (by using exactly the same arguments) also for this extension of Definition \ref{dn} (cfr \cite{bm2})\end{remark}

We use our notion of regularity in order to  proving some splitting criterion on $X=\mathbf P^{n_1}\times\dots \times\mathbf P^{n_s}\times\Q_{m_1}\times\dots \times\mathbf \Q_{m_q}$.\\
\begin{theorem}\label{tpq}Let $E$ be a rank $r$ vector bundle on $X=\mathbf P^{n_1}\times\dots \times\mathbf P^{n_s}\times\Q_{m_1}\times\dots \times\mathbf \Q_{m_q}$ ($m_1,\dots , m_q>2$). Set $d=n_1+\dots+n_s+m_1+\dots +m_q$.\\ Then the following conditions are equivalent:
  \begin{enumerate}
  \item for any $i=1, \dots, d-1$ and for any integer $t$,  $$H^i(E(t,\dots, t)\otimes \E^1_{n_1-k_1}\boxtimes\dots\boxtimes\E^s_{n_s-k_s}\boxtimes\G^1_{m_1-h_{1}}\boxtimes\dots\boxtimes\G^q_{m_q-h_q}))=0$$
   whenever $k_1+ \dots, +k_s+h_1+\dots +h_q=i$, $-n_j\leq k_j\leq 0$ for any $j=1,\dots , s$ and $-m_l\leq -h_l\leq 0$ for any
$l=1,\dots , q$.\\
  \item There are $r$ integer $t_1, \dots, t_r$ such that $E\cong \bigoplus_{i=1}^r \sO(t_i,\dots, t_i)$.
  \end{enumerate}
  \end{theorem}
  \begin{proof}  $(1)\Rightarrow (2)$. Let assume that $t$ is an integer such that $E(t, \dots, t)$ is regular but
  $E(t-1,\dots, t-1)$ not.\\
By the definition of regularity and $(1)$ we can say that $E(t-1,\dots, t-1)$ is not regular if and only if
$$H^{d}(E(t-1, \dots, t-1)\otimes \sO(-n_1,\dots, -n_s,-m_1+1,\dots, -m_q+1))\not=0.$$ By Serre duality we have that $H^0(E^{\vee}(-t,\dots, -t))\not=0$.\\
 Now since $E(t,\dots, t)$ is globally generated by Remark \ref{gg} and $H^0(E^{\vee}(-t,\dots, -t))\not=0$ we can conclude
 that $\sO$ is a direct summand of $E(t,\dots, t)$.\\
      By iterating these arguments we get $(2)$.\\
      $(2)\Rightarrow (1)$. By K\"{u}nneth formula for any $i=1, \dots, m+n-1$ and for any integer $t$, $$H^i(\sO(t,\dots, t)\otimes \E^1_{n_1-k_1}\boxtimes\dots\boxtimes\E^s_{n_s-k_s}\boxtimes\G^1_{m_1-h_{1}}\boxtimes\dots\boxtimes\G^q_{m_q-h_q}))=0$$
   whenever $k_1+ \dots, +k_s+h_1+\dots +h_q=i$, $-n_j\leq k_j\leq 0$ for any $j=1,\dots , s$ and $-m_l\leq -h_l\leq 0$ for any
$l=1,\dots , q$.\\ Then $\sO$ satisfies all the conditions in $(1)$.\end{proof}

\begin{theorem}\label{spi}Let $E$ be a rank $r$ vector bundle on $X=\mathbf P^{n_1}\times\dots \times\mathbf P^{n_s}\times\Q_{m_1}\times\dots \times\mathbf \Q_{m_q}$ ($m_1,\dots , m_q>2$).
 Set $d=n_1+\dots+n_s+m_1+\dots +m_q$.\\ Then the following conditions are equivalent:
  \begin{enumerate}
  \item for any $i=1, \dots, d-1$ and for any integer $t$,  $$H^i(E(t,\dots, t)\otimes \E^1_{n_1-k_1}\boxtimes\dots\boxtimes\E^s_{n_s-k_s}\boxtimes\G^1_{m_1-h_{1}}\boxtimes\dots\boxtimes\G^q_{m_q-h_q}))=0$$
   whenever $k_1+ \dots, +k_s+h_1+\dots +h_q\leq i$, $-n_j\leq k_j\leq 0$ for any $j=1,\dots , s$ and $-m_l\leq -h_l\leq 0$ for any
$l=1,\dots , q$ except when $k_1=\dots =k_s=0$ and $h_l=m_l-1$ for any
$l=1,\dots , q$.\\
  \item $E$ is a direct sum of   bundles $\sO$ and $\sO(0,\dots , 0)\boxtimes\Sigma_*\boxtimes\dots\boxtimes\Sigma_*$ with some balanced twist $(t,\dots , t)$.
  \end{enumerate}
  \end{theorem}
  \begin{proof} 
  $(1)\Rightarrow (2)$.
  Let assume that $t$ is an integer such that $E(t,\dots, t)$ is regular but
  $E(t-1,\dots, t-1)$ not.\\
By the definition of regularity and $(1)$ we can say that $E(t-1,\dots, t-1)$ is not regular if and only if one of the
following conditions is satisfied:
     \begin{enumerate}
     \item[i] $H^{d}(E(t-1, \dots, t-1)\otimes \sO(-n_1,\dots, -n_s,-m_1+1,\dots, -m_q+1))\not=0.$
      \item[ii] $H^{n_1+\dots+n_s+m_1-1+\dots +m_q-1}(E(t-1, \dots, t-1)\otimes \sO(-n_1,\dots, -n_s)\boxtimes\Sigma_*(-m_1+1)\boxtimes\dots\boxtimes\Sigma_*(-m_q+1))\not=0.$
    \end{enumerate}
  Let us consider one by one the conditions:\\
     $(i)$ Let $H^{d}(E(t-1, \dots, t-1)\otimes \sO(-n_1,\dots, -n_s,-m_1+1,\dots, -m_q+1))\not=0$, we can conclude that $\sO(t,\dots ,t)$ is a direct
summand as in the above theorem.\\
$(ii)$ Let $H^{n_1+\dots+n_s+m_1-1+\dots +m_q-1}(E(t, \dots, t)\otimes \sO(-n_1-1,\dots, -n_s-1)\boxtimes\Sigma_*(-m_1)\boxtimes\dots\boxtimes\Sigma_*(-m_q))\not=0$.
 Let us consider the following exact sequences tensored by $E(t, \dots, t)$:

 $$0 \to \sO(-n_1-1,\dots, -n_s-1)\boxtimes\Sigma_*(-m_1)\boxtimes\dots\boxtimes\Sigma_*(-m_q) \rightarrow\dots$$ $$\hskip4cm \dots\rightarrow \sO(0,-n_2-1,\dots, -n_s-1)\boxtimes\Sigma_*(-m_1)\boxtimes\dots\boxtimes\Sigma_*(-m_q)
\to 0,$$
$$0 \to \sO(0,-n_2-1,\dots, -n_s-1)\boxtimes\Sigma_*(-m_1)\boxtimes\dots\boxtimes\Sigma_*(-m_q) \rightarrow \dots$$ $$\hskip4cm\dots\rightarrow \sO(0,0,-n_3-1,\dots, -n_s-1)\boxtimes\Sigma_*(-m_1)\boxtimes\dots\boxtimes\Sigma_*(-m_q)
\to 0,$$
$$\dots$$
$$0 \to \sO(0,\dots,0, -n_s-1)\boxtimes\Sigma_*(-m_1)\boxtimes\dots\boxtimes\Sigma_*(-m_q) \rightarrow \dots$$ $$\hskip4cm\dots\rightarrow \sO(0,\dots, 0)\boxtimes\Sigma_*(-m_1)\boxtimes\dots\boxtimes\Sigma_*(-m_q)
\to 0.$$
Since all the bundles in the above sequences are 
$$\E^1_{n_1-k_1}\boxtimes\dots\boxtimes\E^s_{n_s-k_s}\boxtimes\G^1_{m_1-h_{1}}\boxtimes\dots\boxtimes\G^q_{m_q-h_q}$$ with decreasing indexes, by using the vanishing conditions in $(1)$ we can see that there is a surjection from $$H^{m_1-1+\dots +m_q-1}(E(t, \dots, t)\otimes \sO(0,\dots, 0)\boxtimes\Sigma_*(-m_1)\boxtimes\dots\boxtimes\Sigma_*(-m_q))$$ to $$H^{n_1+\dots+n_s+m_1-1+\dots +m_q-1}(E(t, \dots, t)\otimes \sO(-n_1-1,\dots, -n_s-1)\boxtimes\Sigma_*(-m_1)\boxtimes\dots\boxtimes\Sigma_*(-m_q)).$$
Let us consider now the following exact sequences on $\Q_{m_1}\times\dots \times\mathbf \Q_{m_q}$ for any integer $p$:  $$ 0\rightarrow \Sigma_*(-m_1)\boxtimes\dots\boxtimes\Sigma_*(p-1)\rightarrow \Sigma_*(-m_1)\boxtimes\dots\boxtimes\sO(p)^{2^{([m_q+1/2])}} \rightarrow \Sigma_*(-m_1)\boxtimes\dots\boxtimes\Sigma_*(p)\rightarrow 0.$$ We get the long exact sequence
$$0 \to\Sigma_*(-m_1)\boxtimes\dots\boxtimes\Sigma_*(-m_q) \rightarrow\Sigma_*(-m_1)\boxtimes\dots\boxtimes\sO(-m_q+1)^{2^{([m_q+1/2])}}\rightarrow\dots$$  $$\dots\rightarrow \boxtimes\Sigma_*(-m_1)\boxtimes\dots\boxtimes\Sigma_*(-1)
\to 0.$$ In the same way we can get
$$0 \to\Sigma_*(-m_1)\boxtimes\dots\boxtimes\Sigma_*(-1) \rightarrow\Sigma_*(-m_1)\boxtimes\dots\boxtimes\sO(-m_{q-1}+1)^{2^{([m_{q-1}+1/2])}}\rightarrow\dots$$  $$\dots\rightarrow \boxtimes\Sigma_*(-m_1)\boxtimes\dots\boxtimes\Sigma_*(-1)\boxtimes\Sigma_*(-1)
\to 0,$$
$$\dots$$
$$0 \to\Sigma_*(-m_1)\boxtimes\Sigma_*(-1)\boxtimes\dots\boxtimes\Sigma_*(-1) \rightarrow\sO(-m_{1}+1)^{2^{([m_{1}+1/2])}}\boxtimes\Sigma_*(-1)\boxtimes\boxtimes\Sigma_*(-1)\rightarrow\dots$$  $$\dots\rightarrow \Sigma_*(-1)\boxtimes\dots\boxtimes\Sigma_*(-1)
\to 0.$$
Then on $\mathbf P^{n_1}\times\dots \times\mathbf P^{n_s}\times\Q_{m_1}\times\dots \times\mathbf \Q_{m_q}$ we can obtain the following exact sequence tensored by $E(t, \dots, t)$:
$$0 \to \sO(0,\dots, 0)\boxtimes\Sigma_*(-m_1)\boxtimes\dots\boxtimes\Sigma_*(-m_q)\rightarrow \dots$$ $$\hskip4cm\dots\rightarrow \sO(0,\dots, 0)\boxtimes\Sigma_*(-1)\boxtimes\dots\boxtimes\Sigma_*(-1)
\to 0.$$ By using the vanishing conditions in $(1)$ as above we can see that there is a surjection from $$H^0(E(t, \dots, t)\otimes\sO(0,\dots, 0)\boxtimes\Sigma_*(-1)\boxtimes\dots\boxtimes\Sigma_*(-1))$$ to $$H^{m_1-1+\dots +m_q-1}(E(t, \dots, t)\otimes \sO(0,\dots, 0)\boxtimes\Sigma_*(-m_1)\boxtimes\dots\boxtimes\Sigma_*(-m_q))$$ and we can conclude that $$H^0(E(t, \dots, t)\otimes\sO(0,\dots, 0)\boxtimes\Sigma_*(-1)\boxtimes\dots\boxtimes\Sigma_*(-1))\not=0.$$
This means that there exists a non zero map$$  f: E(t, \dots, t)\rightarrow \sO(0,\dots, 0)\boxtimes\Sigma_*\boxtimes\dots\boxtimes\Sigma_*.$$

On the other hand  $$H^{n_1+\dots+n_s+m_1-1+\dots +m_q-1}(E(t, \dots, t)\otimes \sO(-n_1-1,\dots, -n_s-1)\boxtimes\Sigma_*(-m_1)\boxtimes\dots\boxtimes\Sigma_*(-m_q))\cong $$ $$\cong H^{q}(E(t, \dots, t)\otimes \sO(0,\dots, 0)\boxtimes\Sigma_*(-1)\boxtimes\dots\boxtimes\Sigma_*(-1)).$$ 

Let us consider the following exact sequences tensored by $E^{\vee}(-t, \dots, -t)$:
$$0 \to \sO(0,\dots, 0)\boxtimes\Sigma_*(-1)\boxtimes\dots\boxtimes\Sigma_*(-1)\rightarrow \dots$$ $$\hskip4cm\dots\rightarrow \sO(0,\dots, 0)\boxtimes\Sigma_*\boxtimes\dots\boxtimes\Sigma_*
\to 0.$$

By using the Serre duality and the vanishing conditions in $(1)$ we can conclude that $$H^0(E^\vee(-t,\dots, -t)\otimes\sO(0,\dots, 0)\boxtimes\Sigma_*\boxtimes\dots\boxtimes\Sigma_*)\not=0.$$
This means that there exists a non zero map$$  g:  \sO(0,\dots, 0)\boxtimes\Sigma_*\boxtimes\dots\boxtimes\Sigma_*\rightarrow E(t, \dots, t).$$
Then, by arguing as in \cite{bm2} Theorem $1.2$, we see that the composition of the maps $f$ and $g$ is not zero so must be the identity and we have that $\sO(0,\dots, 0)\boxtimes\Sigma_*\boxtimes\dots\boxtimes\Sigma_*$ is a direct summand of $E(t,\dots, t)$.\\
By iterating these arguments we get $(2)$.\\

$(2)\Rightarrow (1)$. We argue as in Theorem \ref{tpq}. Since $H^i(\Q_n, \Sigma_*(e))\not=0$ if and only if $i=0$ and $e\geq 0$ or $i=n$ and $e\leq -n-1$, we have that  $\sO(0,\dots , 0)\boxtimes\Sigma_*\boxtimes\dots\boxtimes\Sigma_*$ $\sO$ satisfies all the conditions in $(1)$.
  \end{proof}

\bibliographystyle{amsplain}

\begin{thebibliography}{99}
\bibitem
{A} {\sc  V. Ancona and G. Ottaviani},
\emph{Some applications of Beilinson's theorem to projective spaces and quadrics}, 1991, Forum Math. 3, no. 2, 157-176.
\bibitem{bm1} {\sc E. Ballico and F. Malaspina,} \emph{ Qregularity and an Extension of Evans-Griffiths Criterion to Vector Bundles on Quadrics}, 2007, preprint  arXiv:0802.0451v1.
\bibitem{bm2} {\sc E. Ballico and F. Malaspina,} \emph{ Regularity and Cohomological Splitting Conditions for Vector Bundles on Multiprojectives Spaces}, 2008, preprint arXiv:0802.0960.
\bibitem
{Be} {\sc A. A. Beilinson},
\emph{Coherent sheaves on $\mathbb{P}^n$ and Problems of Linear Algebra}, 1979, Funkt. Anal. Appl. 12, 214-216.
\bibitem
{CM0} {\sc L. Costa and R.M. Mir\'o-Roig},
\emph{Tilting sheaves on toric varieties}, 2004, Math. Z. 248,    849-865.
\bibitem
{CM1} {\sc L. Costa and R.M. Mir\'o-Roig},
\emph{Cohomological characterization of vector bundles on multiprojective spaces}, 2005, J. of Algebra, 294,    73-96.
\bibitem{CM3} {\sc L. Costa and R.M. Mir\'o-Roig},
\emph{Monads and regularity of vector bundles on projective varieties}, 2007, Preprint.
\bibitem{hw} {\sc J. W. Hoffman and H. H. Wang,} \emph{Castelnuovo-Mumford
regularity in biprojective spaces,} Adv. Geom. 4 (2004),
no. 4, 513--536.

\bibitem
{Ho} {\sc G. Horrocks},
\emph{Vector bundles on the punctured spectrum of a ring}, 1964, Proc. London Math. Soc. (3) 14, 689-713.
\bibitem
{Ka1} {\sc M. M. Kapranov},
\emph{On the derived category of coherent sheaves on Grassmann manifolds}, 1985, Math. USSR Izvestiya, 24, 183-192.
\bibitem
{Ka2} {\sc M. M. Kapranov},
\emph{On the derived category of coherent sheaves on some homogeneous spaces}, 1988, Invent. Math., 92, 479-508.
\bibitem
{Kn} {\sc H. Kn\"{o}rrer},
\emph{Cohen-Macaulay modules of hypersurface singularities I}, 1987, Invent. Math. 88, 153-164.
\bibitem
{Ml}{\sc F. Malaspina},
\emph{A Few Splitting Criteria for Vector Bundles}, 2007, arXiv:0802.1062, Ricerche di Matematica to appear.
\bibitem{m} {\sc D. Mumford,} \emph{Lectures on curves on an algebraic surface,
Princeton University Press,} Princeton, N.J., 1966.
\bibitem
{Ot1}{\sc G. Ottaviani},
\emph{Criteres de scindage pour les fibres vectoriel sur les grassmanniennes et les quadriques }, 1987, C. R. Acad. Sci. Paris, 305, 257-260.
\bibitem
{Ot2}{\sc G. Ottaviani},
\emph{Spinor bundles on Quadrics}, 1988, Trans. Am. Math. Soc:, 307, no 1, 301-316.
\bibitem
{Ot3}{\sc G. Ottaviani},
\emph{ Some extension of Horrocks criterion to vector bundles on Grassmannians and quadrics}, 1989, Annali Mat. Pura Appl. (IV) 155, 317-341.
\end{thebibliography}

\end{document}